\documentclass[english, 11pt]{amsart} %Xtop

\usepackage{cite}
\usepackage{amsmath}
\usepackage{amsfonts}
\usepackage{amssymb}
\usepackage{ae}
\usepackage{mathrsfs,array,graphicx}
\usepackage[all,ps]{xy}
\usepackage{hyperref}

\usepackage{note}
\usepackage{a4wide}
\usepackage{std_macros}
\renewcommand{\G}{\mathbb{G}}

\title{Existence of cuspidal representations of $p$-adic reductive groups}

\author{Arno Kret}
\begin{document}
\begin{abstract}
We prove that any reductive group $G$ over a non-Archimedean local field has a cuspidal complex representation.
\end{abstract}

\maketitle
%\tableofcontents

\section{Introduction}

%%
%%
%%  Er is nog een trucje om de berekening op het eind bij de orthogonale groep korter 
%%  te maken.
%%
%%

We prove the following Theorem:

\begin{theorem}\label{lemmaB}\label{existcusp}\label{exiscusp}
Let $G$ be a reductive group over $F$. Then $G(F)$ has a cuspidal complex representation. 
\end{theorem}

This theorem is ``folklore'', but we could not find a proof for it in the literature. After some reduction steps the proof  consists of finding certain characters in general position of elliptic maximal tori of $G$. In case the cardinal of the residue field of $F$ is ``large with respect to $G$'', then there are quick arguments to show that characters in general position exist; see for example \cite[lemma 8.4.2]{carterfinitegroupsoflietypebook}. It are the small groups over small fields and big Weyl groups that might cause problems, and in this article we show that such problems do not occur.

\bigskip

\textbf{Acknowledgements:} I wish to thank my thesis adviser Laurent Clozel for helping me putting together this argument and correcting my mistakes. I thank Guy Henniart for explaining me how the problem can be reduced to finding characters in general position of certain maximal tori in finite groups of Lie type, and I thank Francois Digne for explaining me how these characters in general position can be found. 

\bigskip

\textbf{Notations:} Throughout, $F$ is a non-Archimedean local field with residue field $k$ which is of cardinality $q$. Let $p$ be the characteristic of $k$, and fix an algebraic closure $\li k$ of $k$.

\section{Reduction to a problem of classical finite groups of Lie type}

Let $P \subset G(F)$ be a maximal proper parahoric  subgroup with associated reductive quotient $M$ over $k$. We claim that $M(k)$ has an irreducible cuspidal representation $\sigma$. When $\sigma$ is proved to exist, then we may construct a cuspidal representation of $G$ as follows, see \cite{morristamelyramified}, \cite{morrispcuspidal} and \cite{morrisinventiones}. Inflate $\sigma$ to obtain a $P$-representation. We may compactly induce the $P$-representation $\sigma$ to a representation of $G(F)$. This $G(F)$-representation need not be irreducible, but its irreducible subquotients are all cuspidal. Therefore Theorem \ref{lemmaB} reduces to the next proposition. 

\begin{proposition}\label{cuspidalfinite}
Let $G$ be a reductive group over the finite field $k$. The group $G(k)$ has a cuspidal complex representation. 
\end{proposition}
\begin{proof}
We will first reduce to $G$ simple and adjoint.  Consider the morphism $G(k) \to G_\ad(k)$. If $\pi$ is a irreducible representation of $G_\textup{ad}(k)$, then, when restricted to a representation of $G(k)$ it will decompose as a finite direct sum $\pi = \bigoplus_i \pi_i$ of irreducible representations. Recall that $\pi$ is cuspidal if and only if $\uH^0(N(k), V) = 0$ for all rational parabolic subgroups $P \subset G$ with Levi decomposition $P = MN$. The map $G \to G_\ad$ is an isomorphism on its image when restricted to $N$. For any parabolic subgroup $P = MN \subset G_\ad$ the inverse image of $P$ in $G$ is a parabolic subgroup with the same unipotent part. Thus, if $\pi$ is cuspidal as $G_\ad(k)$-representation, then the $\pi_i$ are cuspidal representations of $G(k)$. Therefore, we may assume that $G$ is adjoint. But then $G$ is a product of $k$-simple adjoint groups. If the theorem is true for all the factors, then the theorem is true for $G$. So we may assume that $G = \textup{Res}_{k'/k} G'$ where $G'$ is (absolutely) simple and defined over some finite extension $k'$ of $k$. We have $G(k) = G'(k')$, and under this equality cuspidal representations correspond to cuspidal representations. Therefore, we may assume that $G$ is simple and adjoint. 

The simple reductive groups $G$ over $k$ are classified by their root system. We will distinguish cases between the possible root systems. Let us first assume that the root system of $G$ is exceptional, \ie of the form
${}^3 D_4$, $E_6$, ${}^2 E_6$, $E_7$, $E_8$, $F_4$, ${}^2F_4$, $G_2$ or ${}^2 G_2$.
In Carter's book \cite[\S 13.9]{carterfinitegroupsoflietypebook} one finds for each exceptional group the complete list of its unipotent irreducible complex trace characters. He also mentions for each group how many of these characters are cuspidal. As it turns out, in each of the exceptional cases, this number is $>0$ and so in particular all the exceptional groups have a cuspidal representation. Some of the classical groups do not have cuspidal \emph{unipotent} characters. So unfortunately for those groups we cannot find a cuspidal representation in Carter's list.

It remains to verify Proposition \ref{cuspidalfinite} for the simple adjoint groups $G/k$ which are \emph{classical}. Thus if $G$ is split, then it is of type $A_n, B_n, C_n$ or $D_n$, and if it is non-split, then it is of type ${}^2 A_n$ or ${}^2 D_n$. To do this we will use Deligne-Lusztig theory in Section \ref{findingchars} to reduce the problem to finding characters in general position. In section \ref{splitgroups} we will then verify that all split groups have such a character. In sections \ref{unitarygroups} and \ref{orthogonalgroups} we will then find characters in general position for the remaining non-split root systems. The proof Proposition \ref{cuspidalfinite} will then be complete. 
\end{proof}

\section{Characters in general position}\label{findingchars}

Let $G/k$ be a reductive group with connected center. We will apply results of Deligne-Lusztig \cite{delignelusztig}. Pick $\ell$ a prime number different from $p$. Suppose that we are given the following data: $T \subset G$ a torus and $\theta \colon T(k) \to \lql^\times$ a rational character. Then, to this data Deligne and Lusztig associate a virtual character $R^\theta_T$ of $G(k)$ with $\lql$-coefficients \cite[p. 114]{delignelusztig}. 

%% \begin{definition} We recall very briefly the definition of this character $R^\theta_T$. %For more details see \cite[p. 114]{delignelusztig}. 
%% Choose a Borel $B$ subgroup of $G_{ik}$ containing $T_{ik}$, and let $U$ be the unipotent radical of $B$. 
%% Let $\til X_{T \subset B}$ be the $\lik$-variety such that for every $k$-algebra $A$ we have
%% \begin{equation}\label{magicvariety} \til X_{T \subset B}(A) = \lbr g \in G(A) \,\, | \,\, g^{-1}\Frob_q(g) \in \Frob (U(A_{\lik})) \rbr / U(A_{\lik}) \cap \Frob (U(A_{\lik})). 
%% \end{equation}
%% The group $G(k)$ acts on this variety by left multiplication, so by functoriality the \'etale cohomology $\uH^*_c(\til X_{T \subset B, \et}, \lql)$ carries an $G(k)$-representation. Furthermore, the group $T(k)$ normalizes $U(A_{\lik}) \cap \Frob (U(A_{\lik}))$, so    $T(k)$ acts on the variety in Equation \ref{magicvariety} by right multiplication, and thus also on the \'etale cohomology of this variety. The $G(k)$-action commutes with the $T(k)$-action. 
%% Let $\uH^i_c(\til X_{T \subset B, \et}, \lql)_\theta$ be the subspace of elements $x \in \uH^i_c(\til X_{T \subset B, \et}, \lql)$ on which $T(k)$ acts via $\theta$. By definition $R^\theta_T$ is the trace character of the virtual $G(k)$-representation 
%% $\sum_{i=0}^\infty (-1)^i \uH^i_c(\til X_{T \subset B, \et}, \lql)_\theta$. 
%% \end{definition}
%% 
%% The character $R^\theta_T$ is independent of $B$. 
Let $\sigma(G)$ be the $k$-rank of $G$ and let $\sigma(T)$ be the $k$-rank of $T$. Proposition \cite[Prop. 7.4]{delignelusztig} states that the character $(-1)^{\sigma(G) - \sigma(T)} R^\theta_T$ comes from an actual irreducible $G(k)$-representation $\pi_T^\theta$ if the character $\theta$ is in \emph{general position}, ie  if the rational Weyl group of $T$ acts freely on it.  
 Theorem \cite[thm 8.3]{delignelusztig} states that if, additionally, $T$ is elliptic, %\footnote{\ie is not contained in any proper rational parabolic subgroup of $G$, 
then $\pi_T^\theta$ is cuspidal. Assume for the moment that we have such a pair $(T, \theta)$. Pick an isomorphism $\iota \colon \lql \isomto \C$; then the $G(k)$-representation $\pi_T^\theta \otimes_\iota \C$ is complex cuspidal and irreducible. Therefore, the proof of  Proposition \ref{cuspidalfinite} is reduced to Proposition \ref{computation}, Proposition \ref{computationunitarygroup} and Proposition \ref{computationorthogonalgroup}. 

\section{The split classical groups}\label{splitgroups}

Before continuing with the proof, we recall some generalities. Let $G/k$ be a reductive group. 
Let $(T_0, B_0)$ be a pair consisting of a maximal torus and a Borel subgroup which contains $T_0$, both defined over $k$. Let $W_0/k$ be the Weyl group of $T_0 \subset G$. 
The Frobenius $\Frob_q = (x \mapsto x^q) \in \Gal(\lik/k)$ acts on the root datum % $(X^*(T_{0,\lik}), X_*(T_{0,\lik}), \Phi, \Phi^\vee)$ 
of $G$ 
by a diagram automorphism. By abuse of notation this diagram automorphism is also denoted $\Frob_q$. 

%From the quintet $(X^*(T_{0,\lik}), X_*(T_{0,\lik}), \Phi, \Phi^\vee, \rho)$ one may reconstruct $(G, T_0)$ as simple reductive algebraic group with maximal torus defined over $k$. Recall that the possible quintets are classified. Also recall that our $G$ is simple (and adjoint), so $(X^*(T_{0,\lik}), X_*(T_{0,\lik}), \Phi, \Phi^\vee)$ must be one of the $A_\ell, B_\ell$, \ldots, etc, see \cite[chap 6, \S 4 -- \S 13]{bourbakilie}. Maybe the reader is less familiar with the possibilities for quintets with non-trivial $\rho$. Here the possibilities are
%\begin{equation}\label{diagramautoms}
%{}^2 A_\ell\ (\ell \geq 2), \quad {}^2 D_\ell\  (\ell \geq 3), \quad {}^3 D_4, \quad {}^2 E_6,
%\end{equation}
%where the left exponent denotes the order of the diagram automorphism induced by Frobenius (in \cite[chap 6]{bourbakilie} these automorphisms are defined explicitely).

We can carry out the following construction. Let $\chi \colon T_{0,\lik} \to \G_{m, \lik}$ be a character. Restrict to $T_0(k)$ to get a morphism
$T_0(k) \hookrightarrow T_0(\lik) \to \lik^\times$.
From this construction we obtain a map 
$X^*(T_0) \to \Hom(T_0(k), \lik^\times)$, and this map fits in the exact sequence
\begin{equation}\label{ES}
0 \to X^*(T_0) \overset {\Phi - 1} \to X^*(T_0) \to  \Hom(T_0(k), \lik^\times) \to 0,
\end{equation}
of $\Z[W_0(k)]$-modules (see \cite[\S 5]{delignelusztig}). Here $\Phi$ is the relative $q$-Frobenius of $T_{0, \lik}$ over $\lik$, \ie  given by $f \otimes \lambda \mapsto f^q \otimes \lambda$ 
on the global sections $\cO_{T_0}(T_0) \otimes_{k} \lik$ of $T_{0, \lik}$. Recall that we write $\Frob_q$ for the Frobenius $f \otimes \lambda \mapsto f \otimes \lambda^q$ on $\cO_{T_0}(T_0) \otimes_{k} \lik$.

\begin{definition}
Two elements $w, w'$ in $W_0(\lik)$ are \emph{Frobenius conjugate}, or \emph{$\Frob_q$-conjugate}, if there exists an $x \in W_0(\lik)$ such that $w' = x w \Frob_q(x)^{-1}$. 
\end{definition}

The $G(k)$-conjugacy classes of rational maximal tori in $G_{\lik}$ are parametrized by the Frobenius conjugacy classes of $W_0(\lik)$ in the following manner.
Let $N_0$ be the normalizer of $T_0$ in $G$. We have a surjection from $G(\lik)$ to the set of maximal tori in $G_{\lik}$ by sending $g \in G(\lik)$ to the torus ${}^g T_0 : = g T_0 g^{-1}$. The torus ${}^g T_0 \subset G_{\lik}$ is rational (\ie $\Gal(\lik/k)$-stable) if and only if $g^{-1} \Frob_q(g) \in N_0(\lik)$. 

Assume that we have two elements $g, g' \in G(\lik)$ such that the tori ${}^g T_0$, ${}^{g'} T_0$ in $G_{\lik}$ are rational. Then, $g^{-1} \Frob_q(g)$ and $g^{\prime -1} \Frob_q(g')$ lie in $N_0(\lik)$ so we can map them to elements of the Weyl group $W_0(\lik)$ via the canonical surjection $\pi \colon N_0(\lik) \to W_0(\lik)$. The torus ${}^{g} T_0 \subset G_{\lik}$ is equal to the torus ${}^{g'} T_0 \subset G_{\lik}$ if and only if 
$$
\pi(g^{-1} \Frob_q(g)) \equiv \pi(g^{\prime -1} \Frob_q(g')) \in W_0(\lik)/_{\textup{Frobenius conjugacy}}, 
$$
(for the proof of this fact, see \cite[III.3.23]{dignemichaelbook}). This completes the description how Frobenius conjugacy classes in $W_0(\lik)$ parametrize $G(k)$-conjugacy classes of maximal tori in $G_{\lik}$. 

\begin{notation}
We will write $T_0(w)$ for the torus ${}^g T_0$.  
\end{notation}

\begin{proposition}\label{computation}
Let $G/k$ be a classical simple adjoint group. Then $G$ has an anisotropic maximal torus $T \subset G$  together with a character $\theta \colon T(k) \to \C^\times$ in general position. 
\end{proposition}
%\begin{remark}
%The conditions ``classical'', ``simple'' and ``adjoint'' on $G$ may all be dropped. 
%\end{remark}
\begin{proof}
To prove this proposition we will translate it to an explicit combinatorial problem on Dynkin diagrams. We will then use the classification of such diagrams and calculate to obtain the desired result. 

Let $(T_0, B_0)$ be a pair consisting of a split maximal torus and a Borel subgroup which contains $T_0$, both defined over $k$. Let $w \in W_0(\lik)$ be a Coxeter element and let $T = T_0(w) \subset G$ be the maximal torus corresponding to the Frobenius conjugacy class $\li w \subset W_0(\lik)$ generated by $w$. 

Pick $g \in G(\lik)$ such that $g^{ -1} \Frob_q(g) \in N_0(\lik)$ and $\pi(g^{-1} \Frob_q(g)) = w \in W_0(\lik)$. The conjugation-by-$g$-map $G_{\lik} \to G_{\lik}$ induces an isomorphism from $T_{0,\lik}$ to ${}^g T_{0,\lik} = T_{\lik}$, and in turn an isomorphism $X^*(T) \isomto X^*(T_0)$. 
Under this isomorphism, the Frobenius $\Frob_q$ on $X^*(T)$ corresponds to the automorphism $w \Frob_q$ on $X^*(T_0)$, and similarly $\Phi$ on $X^*(T)$ corresponds to $w \Phi$ on $X^*(T_0)$. 

To see that the torus $T_0(w)$ is anisotropic it suffices to prove that $X^*(T_0(w))^{\Frob_q} = 0$. We will verify this in each individual case below. 
%%The order of the Frobenius $\Frob_q$ acting on $T_0(w)$ is equal to $\# \langle w \rangle$, and therefore equals to the order of the rank of the group. This implies that $T_0(w)$ is anisotropic, and in particular it is elliptic. 

The rational Weyl group $W_T(k)$ of the torus $T$ is equal to the set of those elements $w \in W_T(\lik)$ in the absolute Weyl group whose action on the characters $X^*(T)$ is equivariant for the Frobenius $\Frob_q$. 
Therefore, under the bijection $W_T(\lik) \isomto W_0(\lik)$, the image of $W_T(k)$ in $W_0(\lik)$ is equal to the set of all $t \in W_0(\lik)$ such that $t (w \Frob_q) = (w \Frob_q) t$.
Because $T_0$ is split, the automorphism $\Frob_q$ acts trivially on $X^*(T_0)$. Therefore, the image of $W_T(k)$ in $W_0(\lik)$ is the centralizer of $w \in W_0(\lik)$. Because $w$ is a Coxeter element this centralizer is equal to the subgroup generated by $w\in W_0(\lik)$. 

Choose an embedding of groups $\iota \colon \lik^\times \hookrightarrow \C^\times$. Then, using $\iota$, we may identify $\Hom(T(k), \lik^\times)$ with $\Hom(T(k), \C^\times)$. The set $\Hom(T(k), \C^\times)$ is the set of characters of $T(k)$. We are interested in the subset of $\Hom(T(k), \C^\times)$ consisting of those characters which are in general position. 
Under the bijection
$\frac { X^*(T_0) } { (w \Phi - 1) X^*(T_0) } \isomto \Hom(T(k), \C^\times)$
the action of the group $W_T(k)$ on the right corresponds to the action of the subgroup $\langle w \rangle \subset W_0(\lik)$ on the set on the left. 
The problem of finding an elliptic torus together with a character in general position is thus translated into a problem of the root system of $(G, B_0, T_0)$: Pick any Coxeter element $w$ in the Weyl group of the root system, and find an element $v$ in $\frac{X^*(T_0)}{ (w \Phi - 1) X^*(T_0)}$ which is such that $w^r v \neq v$ for all $r = 1 \ldots h$, where $h = \# \langle w \rangle$ is the Coxeter number of $G$. 

Before starting the computations, let us make the following 3 remarks to clarify. First, the relative $q$-Frobenius $\Phi$ acts on $X^*(T_0)$ by $\chi \mapsto \chi^q$ ($T_0$ is split). And second, because the group $G$ is adjoint, the root lattice of $G$ is equal to the weight lattice $X^*(T_0)$. Finally, the facts on Dynkin diagrams that we state below come from Bourbaki \cite[chap 6, \S 4 -- \S 13]{bourbakilie}.

\medskip

\noindent $\bullet$\quad $G$ is split of type $B_n$ with $n \in \Z_{\geq 2}$. The root system of $G$ may be described as follows. Let $V = \R^n$  with its canonical basis $e_1, \ldots, e_n$ and the standard inner product.  Define
$\alpha_1 = e_1 - e_2, \alpha_2 = e_2 - e_3, \ldots, \alpha_{n - 1} = e_{n - 1} - e_n, \alpha_n = e_n$. The elements $\alpha_1, \ldots, \alpha_n \in \Z^n$ are the simple roots, and the root lattice is equal to $\Z^n \subset \R^n$. The element $w = w_{\alpha_1} w_{\alpha_2} \cdots w_{\alpha_n}$ is a Coxeter element of the Weyl group; it acts on $\R^n$ by $(x_1, \ldots, x_n) \mapsto (-x_n, x_1, \ldots, x_{n - 1})$. It is clear that there are no elements in the root lattice invariant under the action of $w \Frob_q$. This implies that $T_0(w)$ is anisotropic. 

We claim that the element $e_1 \in \Z^n$ reduces to an element of $\Z^n/(w \Phi - 1) \Z^n$ in general position. The order of $w$  is equal to $2 n$, so $\# \stab_{\langle w \rangle}(v)$ divides $2n$. Therefore, it suffices to check that for all $r \in \{1,
 \ldots, n\}$ we have $w^r(e_1) - e_1 \notin (w \Phi - 1) \Z^n$. 
 
 We distinguish cases. Assume first $r = n$. Then $w^r$ acts on $V$ by $v \mapsto -v$. We have $w^n(e_1) - e_1 = (-2, 0, \ldots, 0)$. Assume that we have an $x = (x_1, \ldots, x_n) \in \Z^n$ with $(w \Phi - 1)x = (-2, 0, 0, \ldots, 0)$. Then
 \begin{equation}\label{blmonamib}
-q x_n - x_1 = -2, \ q x_1 - x_2 = 0, \ q x_2 - x_3 = 0, \ \ldots, \ q x_{n - 1} - x_n = 0.
\end{equation}
From this we get $x_n = q^{n - 1} x_1$, and $-2 = -q^n x_1 - x_1 = -(1 + q^n)x_1$ which is not possible. So we have dealt with the case $r = n$.

Now assume that $r \in \{1, \ldots, n - 1\}$. Then $w^r(e_1) - e_1 = e_{r+1} - e_1$. Assume that we have an $x = (x_1, \ldots, x_n) \in \Z^n$ such that
\begin{equation}\label{blmonami}
-q x_n - x_1 = -1, \quad  
q x_r  - x_{r+1} = 1, \quand 
 q x_{i -1}    - x_{i} = 0 \quad (\forall i \notin \{1, r+1\}). 
\end{equation}
We find
$$
x_n = q^{n - r - 1} x_{r + 1} = q^{n - r - 1} (q x_r - 1) = q^{n - r} x_{r} - q^{n - r - 1} = q^{n - 1} x_1 - q^{n - r - 1},
$$
and $x_1 - 1 = -q x_n = -q (q^{n - 1} x_1 - q^{n - r -1}) $, 
which implies
$$
x_1 = \frac {q^{n - r} + 1}{q^n + 1}, 
$$
but $|q^{n - r} - 1|_\infty < |q^n + 1|_\infty$, so $x_1$ is not integral: contradiction. This completes the proof that $e_1 \in X(T)$ is a character in general position in case $G$ is of type $B_n$.  

\medskip
\noindent $\bullet$ \quad  $G$ is split of type $C_n$ with $n \in \Z_{\geq 2}$. 
The root system of $G$ may be described as follows. Let $V = \R^n$  with its canonical basis $e_1, \ldots, e_n$ and the standard inner product.  Define
$\alpha_1 = e_1 - e_2, \alpha_2 = e_2 - e_3, \ldots, \alpha_{n - 1} = e_{n - 1} - e_n, \alpha_n = 2 e_n$. The elements $\alpha_1, \ldots, \alpha_n \in \Z^n$ are the simple roots, and the root lattice $\Lambda$ is equal to the set of $(x_1, \ldots, x_n) \in \Z^n \subset \R^n$ with $\sum_{i=1}^n x_i \equiv 0 \mod 2$. The element $w = w_{\alpha_1} w_{\alpha_2} \cdots w_{\alpha_n}$ is a Coxeter element of the Weyl group; it acts on $\R^n$ by $(x_1, \ldots, x_n) \mapsto (-x_n, x_1, \ldots, x_{n - 1})$.
It is clear that there are no elements in the root lattice invariant under the action of $w \Frob_q$. This implies that $T_0(w)$ is anisotropic. 

We claim that the element $2 e_1 \in \Lambda$ reduces to an element of $\Lambda/(w \Phi - 1)\Lambda$ in general position. It suffices to verify that $w^r (2e_1) - 2e_1 \notin (w \Phi -1) \Lambda$ for all $r \in \{1, \ldots, n\}$. 
Let $x = (x_1, \ldots, x_n) \in \R^n$ be the vector satisfying the equations in Equation \ref{blmonami}. Then the vector $x' := 2x$ satisfies $w^r(2e_1) - 2e_1 = (w \Phi - 1) x'$. 
Therefore, 
$$
x_1' = 2 \cdot \frac {q^{n - r} + 1}{q^n + 1}.
$$
For $q \neq 2$ we have $2|q^{n - r} + 1|_\infty < |q^n + 1|_\infty$, and for $q = 2$ the numerator and denominator are coprime. Therefore $x_1$ is not integral.

\medskip
\noindent $\bullet$\quad $G$ is split of type $A_n$ with $n \in \Z_{\geq 1}$. Consider inside $\R^{n + 1}$ the hyperplane $V$ with equation $\sum_{i =1}^{n + 1} \xi_i = 0$. 
Define
$\alpha_1 = e_1 - e_2, \alpha_2 = e_2 - e_3, \ldots, \alpha_n = e_n - e_{n + 1}$ (simple roots), $\Lambda = \Z^{n + 1} \cap V$ (root lattice), and $w = w_{\alpha_1} w_{\alpha_2} \cdots w_{\alpha_n}$ (Coxeter element). The element %$w_{\alpha_i}$ for $i \in \{1, \ldots, n\}$ act on $\R^{n+1}$ by switching $e_i$ with $e_{i+1}$ and leaving invariant all the other $e_j$ ($j \notin \{i, i+1\}$).
%Therefore, 
$w$ acts on $V \subset \R^{n + 1}$ by rotation of the coordinates: $(x_1, x_2, \ldots, x_n, x_{n + 1}) \mapsto (x_{n + 1}, x_1, x_2, \ldots, x_{n})$. We have $(w \Phi - 1)(x_1, \ldots, x_{n + 1}) = (q x_{n + 1} - x_1, q x_1 - x_2, q x_2 - x_3, \ldots, q x_n - x_{n + 1})$. 
It is clear that there are no elements in the root lattice invariant under the action of $w \Frob_q$. This implies that $T_0(w)$ is anisotropic. 

We claim that the element $v := e_1 - e_{n + 1} \in \Lambda$ reduces to an element of $\Lambda /(w\Phi - 1) \Lambda$ which is in general position.  The order of $w$ equals $n + 1$. Let $r \in \{1, \ldots, n\}$. Suppose for a contradiction that $w^r(v) - v = (e_{r + 1} - e_r) - (e_1 - e_{n + 1}) \in (w\Phi - 1) \Lambda$. Then we have an element $(x_1, \ldots, x_{n + 1}) \in \Lambda$ such that
\begin{align*}
& q x_{n + 1} - x_1 = -1, \ q x_{r-1} - x_{r} = -1,\ q x_{r} - x_{r + 1} = 1,\ q x_n - x_{n + 1} = 1, \cr
& q x_{i - 1} - x_i = 0 \quad (\forall i\notin \{r+1, r, 1, n + 1\}). 
\end{align*}
By substitution we deduce from this
$q^{n + 1} x_{n + 1} = x_{n + 1} - q^{n} - q^{n + 1 - r} + q^{n -r} + 1$.
But, $q^{n + 1} - 1 > q^n + q^{n - r + 1} - q^{n - r} - 1$,
so $x_{n + 1}$ cannot be integral: contradiction. 
%From the last line we deduce that
%$x_{r - 1} = q^{r - 2}x_1$ and $x_n = q^{n - r - 1} x_{r + 1}$. 
%Then, substituting and simplifying this with the first line gives that
%$$
%x_{n + 1} = \frac {1 - q^{n + 2} - q^{n - r + 2}} { q^{n + 3} - 1}.
%$$
%But the right hand side does not lie in $\Z$: 
%$$
%|1 - q^{n + 2} - q^{n - r + 2}|_\infty = q^{n - r + 2} + q^{n + 2} - 1 < 2q^{n + %2} - 1 \leq |q^{n + 3} - 1|_\infty. 
%$$
%This proves that $v$ is a element in general position for the root system $A_n$. 

\medskip
\noindent $\bullet$\quad $G$ is split of type $D_n$ with $n \in \Z_{\geq 4}$. Define
$\alpha_1 = e_1 - e_2$, $\alpha_2 = e_2 - e_3$, $\ldots$, $\alpha_{n - 1} = e_{n -1} - e_{n}$, $\alpha_n = e_{n - 1} + e_n$ (simple roots), $\Lambda$ the set of $(x_1, \ldots, x_n) \in \Z^n$ such that $\sum_{i=1}^n x_i \equiv 0 \mod 2$ (root lattice).

Unfortunately the above procedure to produce anisotropic tori and characters in general position does not work for this group $G$ for the following reason. Let $w = w_{\alpha_1} \cdots w_{\alpha_n}$ be the Coxeter element of the Weyl group which is the product of the reflections in the simple roots. Then $w$ acts on $V$ by
$$
(x_1, x_2, \ldots, x_n)\mapsto (-x_n, x_1, \ldots, x_{n-2}, -x_{n-1}). 
$$
This implies that the vector $(2, \ldots, 2, -2) \in \Lambda$ is stable under the action of Frobenius and thus the corresponding torus is not anisotropic. 

Let $W_0$ be the Weyl group of the system $D_n$. We have a split exact sequence
\begin{equation}\label{cath}
1 \to (\Z^\times)^n_{\det = 1} \to W_0 \to \iS_n \to 1,
\end{equation}
where $\iS_n$ acts on $\Z^n$ via the natural action and an $\eps = (\eps_i) \in (\Z^\times)^n_{\det = 1}$ acts on a vector $e_i \in \Z^n$ of the standard basis by $\eps e_i = \eps_i e_i$. 

Write $n = m +1$. Let $w = (123 \ldots m) \in \iS_n$. Write $t_k \in (\Z^\times)^n$ for the element with $-1$ on the $k$-th coordinate, and with $1$ on all other coordinates. Define $w' = t_nt_m w \in W_0$. We consider the maximal torus $T$ in $G$ of type $w'$. The action of $\Frob_q$ on the character group of this torus is given by
$$
\Z^n \owns (x_1, \ldots, x_m, x_n) \mapsto (x_m, x_1, \ldots, -x_{m-1}, -x_n).
$$
We see that there are no non-zero vectors in $\Z^n$ which are invariant under this action. Therefore the torus $T$ is anisotropic. 

The rational Weyl group of $T$ is the set of $s \in W_0$ which commute with $w'$. Let us compute this group. Write $\varphi \colon W_0 \surjects \iS_n$ the natural surjection (see Equation \ref{cath}). Let $s \in W_T(k)$, then $w = \varphi(w') = \varphi(s w s^{-1})$. Therefore $\varphi(s)$ commutes with $w$. This implies that $\varphi(s)$ is a power of $w$. Write $s = \eps w^k$ for some $\nu \in (\Z^\times)^n_{\det = 1}$. We have
$$
s t_n s^{-1} = t_n, \quad s t_m s^{-1} = t_{w^k(m)}, \quad \textup{and} \quad sws^{-1} \eps w^k w w^{-k} \eps = \eps w \eps. 
$$
Therefore
$$
s(w')s^{-1} = s (t_n t_m w) s^{-1} = t_n t_{w^{k}(m)} \eps w\eps, 
$$
which is equivalent to
$$
(\eps_{w(i)} \eps_i) = t_{w^k(m)} t_m. 
$$
A priori there are $4$ solutions $\eps \in (\Z^\times)^n$ of this equation. When we add the condition $\det(\eps) = 1$, then precisely $2$ of those solutions remain. 

Let $\eps \in (\Z^\times)^n_{\det = 1}$ be such that $(\eps_{w(i)} \eps_i) = t_{w(m)} t_m$. We have an exact sequence
$$
1 \to \{1, \nu\} \to W_T(k) \to \langle \eps w \rangle \to 1, 
$$
where $\nu \in (\Z^\times)^n_{\det = 1}$ is given by $\nu_i = -1$ for $i \leq m$ and $\nu_n = (-1)^m$. 

We claim that $v = 2e_m \in \Lambda$ reduces to an element of $\Lambda / (w' \Phi - 1) \Lambda$ in general position. Assume that
$$
(w' \Phi - 1)(x_1, \ldots, x_m, x_n) = (qx_m - x_1, qx_1 - x_2, \ldots, -qx_{m-1} - x_m, -qx_n - x_n).
$$
We ignore the last coordinate, and only work with the vector $(x_1, \ldots, x_m)$. By substitution we deduce that $q^m x_m = -2 - x_m \pm 2 q^{m-r}$. This implies 
\begin{equation}\label{POI}
x_m = 2\frac{ q^{m - r} \pm 1}{q^m + 1}. 
\end{equation}
For $(q,r) \neq (2,1)$ we have $|q^m + 1|_\infty > 2|q^{m-r} \pm 1|_\infty$, and for $(q,r)=(2,1)$ the numerator and denominator have a gcd which divides $3$, so then $q^m + 1 = 3$ and we must have $n = 1$, but we assumed $m \geq 2$. Therefore $x_n$ is not integral. 
\end{proof}

\section{The unitary groups}\label{unitarygroups}

\begin{proposition}\label{computationunitarygroup}
Let $n \geq 3$. 
The simple adjoint group over $k$ with root system ${}^2A_{n-1}$ has an anisotropic maximal torus $T$ together with a character $T(k) \to \C^\times$ in general position. 
\end{proposition}
\begin{proof}
Let $E \subset \lik$ be the quadratic extension of $k$, and let $\sigma \colon E \isomto E$ be the unique non-trivial $k$-automorphism of $E$. The unitary group $\U_n$ over $k$ is the group of matrices $g \in \Res_{E/k} \Gl_{n,E}$ such that $\sigma(g)^\textup{t}g = 1$. 
The adjoint group $\U_{n, \ad}$ of $\U_n$ is the group $\PSU_n$ and this group has root system ${}^2 A_{n-1}$. 

We will distinguish cases between $n$ odd and $n$ even. Assume first that $n$ is odd. 
Let $T_0$ be the torus $(U_1)^n$ embedded diagonally in $U_n$. Then $\Frob_q$ acts on $X^*(T_0)$ by $x \mapsto -x$. We have $X^*(T_0) = \Z^n$ and under this equality, the Weyl group $W_{T_0}(\lik)$ is identified with $\iS_n$. Let $w = (123\ldots n) \in \iS_n = W_{T_0}(\lik)$ and let $T$ be the torus $T_0(w)$. The relative Frobenius $\Phi$ acts on $X^*(T) = \Z^n$ by 
\begin{equation}\label{frobeniusaction}
(x_1, \ldots, x_n) \mapsto (-x_n, -x_1, \ldots, -x_{n-1}).
\end{equation}

We claim that the torus $T$ is anisotropic over $k$. To see this, let $x = (x_1, \ldots, x_n) \in \Z^{n}$ be $w\Frob_q$-invariant. Then
$(x_1, \ldots, x_n) = (-x_n, -x_1, \ldots, -x_{n-1})$,
it follows $x_1 = (-1)^n x_1$, and because $n$ is odd, this implies $x_1 = 0$. The same argument applies to the other $x_i$, and therefore $x = 0$. We proved $X^*(T)^{\Frob_q} = 0$, and thus $T$ is anisotropic. 

The center of $U_n$ is equal to $U_1$ embedded diagonally. Let $\Tad$ be the image of the torus $T$ in the adjoint group of $\U_n$. Then $\Lambda = X^*(T_\ad)$ is the subset of $\Z^n$ consisting of those vectors $x \in \Z^n$ such that $\sum_{i=1}^n x_i = 0$. The Weyl group is $\iS_n$ and it acts on $\Lambda$ via the restriction of the natural action $\iS_n \cal \Z^n$ to $\Lambda$. The rational Weyl group $W_{\Tad}(k) \subset \iS_n$ is the set of elements $w$ commuting with $\Frob_q$. The rational Weyl group is equal to $\langle w \rangle \subset \iS_n$ because all elements of the Weyl group commute with $-1$.

To find an element in general position we must find a vector $v \in \Lambda$ which is such that $w^r(v) - v \notin (\Phi - 1)\Lambda$ for all $r = 1\ldots n-1$. We claim that $v = e_1 - e_n \in \Lambda$ is such a vector.

Assume for a contradiction that $(w\Phi - 1)x = w^r v - v$ for some $x \in \Lambda$. Then
$$
(-qx_n - x_1, -qx_1 - x_2, \ldots, -qx_{n-1} - x_{n}) = (e_{r+1} - e_r) - (e_{1} - e_{n}).
$$
By substitution we deduce from this
$(-q)^{n} x_{n} = x_{n} - (-q)^{n-1} - (-q)^{n - r} + (-q)^{n -1 -r} + 1$, and thus
$$
x_n = -\frac{  (-q)^{n-1} + (-q)^{n - r} - (-q)^{n -1 -r} - 1}{(-q)^n - 1} \in \Z.
$$
We show that this is not possible. 
We will distinguish cases. Assume first that the pair $(q,r)$ is such that the inequality $|(-q)^r + (-q) + 1| < q^{r+1} - 2$ holds. We may then estimate
\begin{align*}
|(-q)^{n-1} + (-q)^{n - r} - (-q)^{n -1 -r} - 1|_\infty &= |(-q)^{n-1-r}((-q)^r + (-q) - 1) - 1| \cr 
&\leq  q^{n-1-r} \cdot | (-q)^r + (-q) - 1 | + 1 \cr
%q^{n-1-r} \cdot | (-q)^r + (-q) - 1 | + 1 
&< q^n - 2 q^{n-1-r} + 1 \leq q^n - 1 \leq |(-q)^n - 1|_\infty.
\end{align*}
This proves that $x_n$ cannot be integral. 

Let us determine the pairs $(q,r)$ for which the above inequality is not true. We have $|(-q)^r + (-q) + 1| \leq q^r + q + 1$.
The inequality $q^r + q + 1 < q^{r+1} - 2$ does not hold for $(q,r) \in \{(2,1), (2,2), (3,1)\}$. To see that it holds in all other cases, observe first that if the inequality holds for $(q,r)$ then it holds also for $(q, r+1)$. By direct verification we see that it holds for $(2,3)$, $(3,2)$, and for $(q, 1)$ in case $q > 3$. 

For $(q, r) \in \{(2,2), (3,1)\}$ we have the inequality $|(-q)^r + (-q) + 1| < q^{r+1} - 2$, so the above proof also applies to these cases. 
In case $(q, r) = (2, 1)$, then we obtain $x_1 = -1 + \frac{(-2)^{n-2}}{(-2)^n - 1}$, which is not integral. 
This completes the proof for $n$ odd. 

Now assume that $n$ is even. Write $n = m + 1$, so that $m$ is odd. Let $T_0 \subset U_n$ be the torus $(U_1)^n$ embedded on the diagonal of $U_n$. Let $w = (123\ldots m) \in \iS_n = W_{T_0}(\lik)$, and consider the torus $T := T_0(w)$. We have $X^*(T) = \Z^n$ on which the Frobenius acts by $-w$. 
The rational Weyl group $W_T(k) \subset \iS_n$ is the set of $s \in \iS_n$ which commute with $-w$. Therefore $W_T(k) = \langle w \rangle$, and in particular $W_T(k) \subset \iS_m$. Let $\Tad$ be the image of the torus $T$ in the adjoint group of $\U_n$. The lattice $X^*(\Tad) = \Lambda \subset \Z^n$ is the set of vectors $(x_1, \ldots, x_n)$ with $\sum_{i = 1}^n x_i = 0$. 
The tori $T$ and $\Tad$ are anisotropic. 
Write $T = T_m \times U_1$, where the torus $T_m$ is the maximal torus in the group $U_m$ that we considered in the odd case. The rational Weyl group $W_T(k)$ preserves this decomposition of $T$.
We have the map $X^*(T_{m, \ad}) \to X^*(\Tad)$, $(x_1, \ldots, x_m) \mapsto (x_1, \ldots, x_m, 0)$. This map is $\iS_m$-equivariant, and it induces a map
$$
\frac {X^*(T_{m, \ad})}{(w\Phi - 1) X^*(T_{m, \ad})} \injects \frac {X^*(T_\ad)}{(w\Phi - 1)X^*(\Tad)},
$$
which is $w$-equivariant. 
Therefore characters in general position are send to characters in general position. By the argument above we know that $T_m$ has characters in general position, so this completes the proof for $n$ even. 
%% 
%% Let $\chi \colon T_m(k) \to \C^\times$ be a character in general position. Via the surjection $T \surjects T_m$ we may extend this character to a character $\chi_T$ of $T$. Because the rational Weyl group of $T$ fixes $T_m$ we see that $\chi_T$ is in general position. 
%% 
%% The rational Weyl group of $T$ fixes $T_m$. 
%% 
%% We claim that the element $v = e_1 - e_n \in \Lambda$ reduces to an element of $\Lambda/(w \Phi - 1) \Lambda$ in general position. Assume for a contradiction that we have an element $(x_1, \ldots, x_n) \in \Lambda$ such that $v - w^r(v) = e_1 - e_m - (e_{m + 1 - r} - e_{m - r}) \in (w\Phi - 1)\Lambda$ for some $r \in \{1, \ldots, m-1\}$. 
%% 
%% Observe that the argument above in the odd case from Equation \ref{unitaryoddargument} onwards does not use that $n$ is odd. Also, observe that the argument proves something slightly more stronger than stated: It proves that there exist no $x \in \Z^n$ which satisfies the equations in Equation \ref{unitaryoddargument} (instead of 
%% $x \in \Lambda$ as stated above Equation \ref{unitaryoddargument}). 
%% Therefore, we may consider the vector $(x_1, \ldots, x_m) \in \Z^m$. By construction, this vector satisfies the equations in Equation \ref{unitaryoddargument} with $n$ replaced by $m$, and we find a contradiction.
%% 
%% This completes the proof.
\end{proof}

\section{The non-split orthogonal groups}\label{orthogonalgroups}

\begin{proposition}\label{computationorthogonalgroup}
Let $n \in \Z_{\geq 4}$. The simple adjoint group $G$ over $k$ with root system ${}^2 D_{n-1}$ has a maximal torus $T \subset G$ with a character $T(k) \to \C^\times$ in general position. 
\end{proposition}
\noindent {\textit{Proof.}} \quad 
Let $J$ be the $2n\times 2n$-matrix consisting of the blocks $\vierkant \ 11\ $ on the diagonal, and all other entries $0$. The group $\OO_{2n}$ over $k$ is the set of matrices $g \in \Gl_{2n,k}$ which are such that $g^\textup{t} J g = J$. The group $\SO_{2n}$ is the group of matrices $g \in \OO_{2n}$ such that  $\det(g) = 1$. The non-split form $\SO_{2n}'$ over $k$ is obtained from $\SO_{2n}$ by twisting the action of $\Frob_q$ with the matrix $s \in \Gl_{2n}$ consisting of the blocks $\vierkant 1\ \ 1$ on the diagonal, except for the last block on the diagonal which is $\vierkant \ 11\ $. This corresponds to replacing the matrix $J$ with the matrix $s J s^{-1} = sJs$ in the definition of the orthogonal group. 

In characteristic $p \neq 2$, the group $\SO_{2n}$ (resp. $\SO_{2n}$) is connected and has root system $D_{n-1}$ (resp. ${}^2 D_{n-1}$). For $p=2$ it is the connected component of identity, $\SO_{2n}^\circ$ (resp. $\SO_{2n}^{\prime \circ}$), that has root system $D_{n-1}$ (resp. ${}^2D_{n-1}$). 

The torus $(\SO_{2}^\circ)^n$ on the diagonal in $\SO_{2n}^\circ$ is a maximal torus, and the torus $T_0 = (\SO_2^\circ)^{n-1} \times U_1$ is a maximal torus of $\SO_{2n}^{\prime \circ}$. We have $X^*(T_0) = \Z^n$ and $\Frob_q$ acts on $X^*(T_0)$ by $(x_1, \ldots, x_n) \mapsto (x_1, \ldots, x_{n-1}, -x_n)$. Let $W_0$ be the absolute Weyl group of $T_0$. We have a split exact sequence
\begin{equation}\label{yas}
1 \to (\Z^{\times})^n_{\det = 1} \to W_0 \to \iS_n \to 1,
\end{equation}
where $\iS_n$ acts on $\Z^n$ by permuting the standard basis vectors, and where an $\eps = (\eps_i) \in (\Z^{\times})^n_{\det = 1}$ acts on a vector $e_i \in \Z^n$ of the standard basis by $\eps e_i = \eps_i e_i$. 

Let $w \in \iS_n \subset W_0$ be the $n$-cycle $(123 \ldots n)$ and consider the torus $T := T_0(w)$. Then $X^*(T) \isomto \Z^n$ via which the action $\Frob_q \cal X^*(T)$ corresponds to the action of $w \Frob_q$ on $\Z^n$. 

We verify that this torus is anisotropic. Let $x = (x_1,\ldots, x_n) \in X^*(T_0)^{w\Frob_q}$. Then
$$
(x_1, \ldots, x_n) = w\Frob_q (x_1, \ldots, x_n) = (-x_n, x_1, \ldots, x_{n-1}),
$$
which implies $x = 0$. Therefore $T$ is anisotropic. 

The rational Weyl group $W_T(k) \subset W_0$ is the set of $s \in W_0$ which commute with $w \Frob_q$. Let us determine this group. Write $\varphi$ for the map $W_0 \surjects \iS_n$ (see Equation \ref{yas}), and write $t_j \in (\Z^\times)^n$ for the element with $-1$ on the $j$-th coordinate, and with $1$ on all other coordinates. Then $\Frob_q = t_n$. 

 If $s \in W_t(k)$, then $s (w t_n) s^{-1} = w t_n$. We apply $\varphi$ to this equality to obtain $w = \varphi(w) = \varphi(s w s^{-1})$, and thus $\varphi(s) \in \iS_n$ commutes with $w$. This implies that $\varphi(s)$ is a power of $w$. Write $s = \eps w^k$ where $\eps \in (\Z^\times)^n_{\det = 1}$. We have
$$
s t_n s^{-1} =\eps w^k (t_n)  w^{-k} \eps = t_{k},
$$
and
$$
s w s^{-1} = \eps w^k w w^{-k} \eps = \eps w \eps. 
$$
Therefore, 
$$
wt_n = s (wt_n) s^{-1} = \eps w \eps \cdot t_{k}. 
$$
This is equivalent to,
$$
\eps w^{-1} \eps w = t_{k} t_n.
$$
Write $\eps = (\eps_i) \in (\Z^\times)^n_{\det = 1}$. Then we have, 
\begin{equation}\label{yacd}
\eps w^{-1} \eps w = (\eps_i) \cdot (\eps_{w(i)}) = (\eps_{i} \eps_{i + 1}) = t_{k} t_n.
\end{equation} 
We will now distinguish cases between $n$ is odd and $n$ is even. Assume first that $n$ is odd. Return to Equation \ref{yacd}, we have $(\eps_i \eps_{i+1}) = t_k t_n$. After the choice of $\eps_n$, the $\eps_i$ for $i < n$ are uniquely determined by this equation. If $\eps$ is one of the solutions, then $-\eps$ is the other solution. We have $\det(-\eps) = (-1)^n \det(\eps) = -\det(\eps)$. Therefore, precisely one of the two solutions has determinant $1$. 
We conclude that the rational Weyl group $W_T(k)$ is equal to $\langle \eps w\rangle$, where $\eps \in (\Z^\times)^n_{\det = 1}$ is the unique element such that $(\eps_i \eps_{i + 1}) = t_1 t_n$. 

Let $\SO_{2n, \ad}^\circ$ be the adjoint group of $\SO_{2n}^\circ$ and let $\Tad$ be the image of the torus $T$ in $\SO_{2n, \ad}^\circ$. Then $X^*(\Tad) \subset X^*(T) = \Z^n$ is the sublattice of elements $(x_1, \ldots, x_n) \in \Z^n$ such that $\sum_{i=1}^n x_i = 0$. 

We claim that the element $v = 2e_n \in \Lambda$ reduces to an element $\li v \in \Lambda / (w\Phi - 1)\Lambda$ in general position. We have
$(\eps w)^r 2e_n = \pm 2 e_{n-r}$, 
for all $r = 1, \ldots, n-1$. We left the sign unspecified, but we mention that it depends on $r$.

Suppose that there exists an $x = (x_1, \ldots, x_n) \in \Lambda$ such that 
$$
(w\Phi - 1)(x_1, \ldots, x_n) = (-qx_n, qx_1, \ldots, qx_{n-1}) - (x_1, \ldots, x_n) = 2e_n \pm 2 e_{r}. 
$$
This implies $-q^n x_n = -2q^{n-r} \pm 2 + x_n$, 
and thus
$$
x_n = 2\frac{ q^{n - r} \pm 1}{q^n + 1}. 
$$
%\end{equation}
%For $(q,r) \neq (2,1)$ we have $|q^n + 1|_\infty > 2|q^{n-r} \pm 1|_\infty$, and for $(q,r)=(2,1)$ the numerator and denominator have a gcd which divides $3$, so then $q^n + 1 = 3$ and we must have $n = 1$, but we assumed $n \geq 4$. Therefore $x_n$ is not integral. 
We have already verified in Equation \ref{POI} that $x_n$ cannot be integral. 
This completes the proof for $n$ odd.

Assume now that $n$ is even. We have $-1 \in (\Z^\times)^n_{\det = 1}$ in case $n$ is even. In Equation \ref{yacd} we found that $(\eps_i \eps_{i+1}) = t_k t_n$. We obtain from this, $\eps_i = \eps_1$ for $i \leq k$ and $\eps_i = -\eps_1$ for $i > k$. Therefore $\det(\eps) = (-1)^k$, independently of $\eps_1$. 
Therefore, $\eps \in (\Z^\times)^n_{\det = 1}$ only if $k$ is even, and if this is the case, then the equation $(\eps_i \eps_{i+1}) = t_k t_n$ has exactly $2$ solutions for $\eps \in (\Z^\times)^n_{\det = 1}$. 

We conclude that $\# W_T(k) = n$, but the group is not cyclic: Pick an $\eps \in (\Z^\times)^n_{\det = 1}$, such that $(\eps_i \eps_{i+1}) = t_2 t_n$ holds. Then the rational Weyl group $W_T(k)$  is equal to $\Z^\times \times  \langle \eps w^2 \rangle$. We have
$$
\forall \nu_1 \in \Z^\times \exists \nu_2 \in \Z^\times:  \quad\quad 
\nu_1 (\eps w^2)^r(2e_n) = \nu_2 2e_{2r}.
$$  
We show that $2 e_n \pm 2e_{n - 2r} \notin (w \Phi - 1) \Lambda$ for all $r = \lbr 1, \ldots, \frac n2\rbr$, and all signs, so that the element $\li {2e_n} \in \Lambda/ (w \Phi - 1) \Lambda$ is in general position. Assume for a contradiction that $x \in \Lambda$ is such that %$(w \Phi - 1) x = 2 e_n \pm 2e_{n - 2r}$.  We have
$$
(w \Phi - 1)(x_1, \ldots, x_n) =  2 e_n \pm 2e_{2r}. 
$$
Then we may proceed as in Equation \ref{POI} to find that $x_n$ is not integral. 

All possible cases are now verified and the proof of Theorem \ref{existcusp} is completed. $\square$

%\addcontentsline{toc}{chapter}{References}
\bibliographystyle{plain}
\bibliography{grotebib}

\bigskip

\noindent Arno Kret \\
\noindent Universit\'e Paris-Sud, UMR 8628, Math\'ematique, B\^atiment 425, \\
\noindent F-91405 Orsay Cedex, France

\end{document}